\documentclass[11pt]{amsart}
\input epsf
\usepackage{latexsym}
\usepackage{graphicx}
\usepackage{amssymb,amsmath,amsthm,amscd,amssymb,enumerate, verbatim}
\usepackage[all]{xy}

\numberwithin{equation}{section}
\newtheorem{cor}[equation]{Corollary}

\newtheorem{thm}[equation]{Theorem}

\newtheorem{add}[equation]{Addendum}

\newtheorem{Example}[equation]{Example}
\newenvironment{ex}{\begin{Example}\rm}{\end{Example}}
\newtheorem{Notation}[equation]{Notation}

\newtheorem{remark}[equation]{Remark}
\newenvironment{rmk}{\begin{remark}\rm}{\end{remark}}
\def\co{\colon\thinspace}

\newcommand{\Int}{\mbox{Int}}

\newcommand{\cd}{\mbox{cd}}

\newcommand{\e}{\varepsilon}

\def\a{\alpha}
\def\G{\Gamma}
\def\g{\gamma}

\def\b{\beta}

\def\d{\partial}

\def\s{\sigma}

\def\S1{\bf S^1}

\begin{document}

\title[Non-aspherical ends and nonpositive curvature]
{\bf Non-aspherical ends and nonpositive curvature}
\thanks{\it 2000 Mathematics Subject classification.{\rm\ 
Primary 53C20.}
{\it\! Keywords:\rm\ aspherical manifold, nonpositive curvature, end}.}\rm

\author{Igor Belegradek \and T. T$\hat{\mathrm{a}}$m Nguy$\tilde{\hat{\mathrm{e}}}$n Phan}

\address{Igor Belegradek\\School of Mathematics\\ Georgia Institute of
Technology\\ Atlanta, GA 30332-0160}\email{ib@math.gatech.edu}
\address{T. T$\hat{\mathrm{a}}$m Nguy$\tilde{\hat{\mathrm{e}}}$n Phan\\
Department of Mathematics\\
Binghamton University\\
State University of New York\\
Binghamton, NY 13902}
\email{tam@math.binghamton.edu}

\date{}

\begin{abstract}
We obtain restrictions on the boundary of a compact 
manifold whose interior admits a complete Riemannian metric 
of nonpositive sectional curvature.
\end{abstract}

\maketitle

\section{Introduction}

In this paper all manifolds are smooth, 
all metrics are Riemannian,  
the phrase ``nonpositive sectional curvature'' 
is abbreviated to ``$K\le 0$''. 
A manifold is {\it covered by} $\mathbb R^n$ if its universal cover 
is diffeomorphic to $\mathbb R^n$, e.g. any complete $n$-manifold of $K\le 0$
is covered by $\mathbb R^n$ by the Cartan-Hadamard theorem. 
A boundary component $B$ of a manifold $N$ is 
{\it $\pi_i$-incompressible\,} if the inclusion $B\hookrightarrow N$ 
induces an injection of $i$th homotopy groups,
and $B$ is {\it incompressible\,} if it is $\pi_i$-incompressible for each $i$.
Thus an incompressible boundary component of an aspherical manifold is
aspherical. Here is the main result of this paper:

\begin{thm}
\label{thm-intro: no injrad}
Let $M$ be a closed connected $m$-manifold that is either infranil or
locally symmetric irreducible of $K\le 0$ and real rank $\ge 2$.
If $N$ is a connected $n$-manifold with compact connected
boundary such that 
$\mathrm{Int}\,N$ admits a complete metric $g$ of $K\le 0$ and $\pi_1(\d N)\cong\pi_1(M)$, then:
\vspace{1pt}
\newline\phantom{\ \ }\textup{(1)}
$\d N$ is incompressible if either\vspace{1pt}
\newline\phantom{\textup{(1)}\ \ }\textup{(1a)} 
$\d N$ has a neighborhood $U$ in $N$ such that $\mathrm{Vol}_g (U\cap\mathrm{Int}\,N)$ 
is finite, \newline\phantom{\textup{(1)}\ \ (}or 
\newline\phantom{\textup{(1)}\ \ }\textup{(1b)} 
$\pi_1(M)$ has no proper torsion-free quotients and
$n-m=1$.
\vspace{2pt}
\newline\phantom{\ \ }\textup{(2)}
If $n-m\ge 3$,
then $\d N$ is $\pi_1$-incompressible, and
$\d N$ is the total 
space \newline\phantom{\textup{(2)}\ \ } of a spherical fibration 
over $M$.
\end{thm}

Recall that a closed manifold is {\it infranil} if it is of the 
form $G/\G$,
where $G$ is a connected, simply-connected, nilpotent Lie group,
$C$  is a maximal compact subgroup of $\mathrm{Aut}(G)$, 
and $\G$ is a torsion-free lattice in $G\rtimes C$.
Infranilmanifolds are aspherical, and their fundamental groups 
are precisely the finitely generated torsion-free 
virtually nilpotent groups.

The assumption ``$\pi_1(M)$ has no proper torsion-free quotients'' in (1b)
holds if $M$ is not infranil (by the Margulis Normal Subgroup Theorem). 
There are also some infranil manifolds
whose fundamental group has no proper torsion-free quotients; 
such manifolds are known in dimensions $2$,
$p^2-1$ and $p^3-p$ for every prime $p$ (see Example~\ref{ex: tf general}).

We actually prove somewhat stronger results as follows:
\begin{itemize}
\item 
The claim (1) holds when the group $\pi_1(\d N)$ is reductive,
where a group $G$ is 
{\it reductive\,} if for any epimorphism
of $G\to H$, where $H$ is a discrete, non-cocompact,
torsion-free isometry group of 
a Hadamard manifold, the group $H$ stabilizes
a horoball or acts cocompactly on a totally geodesic,
proper submanifold. 
See Examples~\ref{ex: anti-A_n} and~\ref{ex: elem}
for a list of reductive groups.
\item
(1a) is true when $\d N$ has a compact neighborhood $U$ in $N$
such that $U\cap\mathrm{Int}\, N$ has ${Inj}\,{Rad}\to 0$, 
see Corollary~\ref{cor-intro: tg}. 
\item
The fibration $\d N\to M$ in (2) is not arbitrary, e.g. 
if the fibration admits {\it no\,} section, then it is isomorphic to a linear sphere bundle over $M$,
and moreover $M$ admits a metric of $K\le 0$,
see Theorem~\ref{thm: spher fibr over M}(2).
\end{itemize}

We say that $B$ {\it bounds} $N$ if $B$ and $N$ are connected 
(not necessarily compact)
manifolds such that $B$ is diffeomorphic to $\d N$; 
this terminology is non-standard.

Gromov~\cite{Gro-rand03} and Izeki-Nayatani~\cite{IseNay05}
constructed many a group with finite classifying space such that
any isometric action of the group on a Hadamard manifold
fixes a point. As we note in Section~\ref{sec: known obstr}, 
every group with these properties is the fundamental group of a 
closed non-aspherical manifold that \vspace{3pt}\newline
\phantom{\ \ } (i) bounds no manifold
whose interior has a complete metric of $K\le 0$, and\vspace{2pt}
\newline \phantom{\ \,} (ii)
bounds a compact manifold whose interior
is covered by $\mathbb R^n$.

Theorem~\ref{thm-intro: no injrad} also gives many 
closed non-aspherical manifolds satisfying (i)--(ii), e.g. 
if $B$ is the total space 
of a linear $S^k$ bundle over $M$,
where $k\ge 2$ and $M$ is as in Theorem~\ref{thm-intro: no injrad}, 
then $B$ satisfies (ii) and the above discussion implies:
\begin{itemize}
\item
If $M$ is infranil and not flat, and if
the $S^k$ bundle $B\to M$ has no section, then
$B$ satisfies (i), see Theorem~\ref{thm: spher fibr over M}(2). 
\item
$B$ cannot bound a manifold that admits
a finite volume complete metric of $K\le 0$, 
see Corollary~\ref{cor-intro: tg}. 
\end{itemize}
Finiteness of volume in the last example is essential: 
the interior of any linear disk bundle over a 
closed manifold of $K\le 0$
admits a complete metric of $K\le 0$~\cite{And-vb}. 

Other ways to produce closed manifolds satisfying (i)--(ii) 
are described Examples~\ref{ex: spherical fibrations}(2)--(3) 
and~\ref{ex: attach circle times homology sphere}. 

By contrast, one does not know whether a
closed {\it aspherical\,} manifold $B$ can
satisfy (i). A basic difficulty is that 
$B\times (0,1)$ could admit a complete metric of $K\le 0$.
Such a metric exists if e.g. $B$ itself admits a metric of $K\le 0$, 
or if $B$ is infranil~\cite{BK-GAFA}. 
In fact, $B$ might even bound a manifold whose interior 
admits a {\it finite volume\,} complete metric of $K\le 0$,
which happens in many cases, 
see~\cite{Ont-smth-hyp, Pha-sol, Bel-ewp}.

There seems to be no simple description of closed manifolds
that bound aspherical ones and in Section~\ref{sec: asph} we review 
some obstructions and examples. 
Results generalizing the parts 
(2), (1a), (1b)
of Theorem~\ref{thm-intro: no injrad}, 
are discussed in Sections~\ref{sec: known obstr},
\ref{sec: new obstr},
\ref{sec: codim 1}, respectively.
Section~\ref{sec: proofs} contains a proof of 
Theorem~\ref{thm-intro: geom}, the main technical result
of this paper, which implies (1).

{\bf Acknowledgments} Belegradek is grateful for NSF support (DMS-1105045).

\section{Boundaries of aspherical manifolds}
\label{sec: asph}

Before imposing any curvature restrictions we study
topological properties of boundaries of aspherical manifolds. 
Allowing $B$ to be noncompact almost makes this task meaningless:
if $B$ bounds $N$, then any open connected 
subset $U$ of $B$ bounds $U\cup\mathrm{Int}\,N$. 
As we see below aspherical manifolds with compact 
non-aspherical boundaries also exist in abundance.

\begin{ex}
Fiber bundles of aspherical manifolds is aspherical,
but their boundary is often not incompressible, e.g.
this happens for the product of two compact aspherical manifolds
with nonempty boundaries (as is immediate from the cohomological 
dimension count).
\end{ex}

\begin{ex}
Complete, locally symmetric manifold of $K\le 0$ that have
finite volume and $\mathbb Q$-rank $\ge 3$ are interiors 
of compact manifolds with
non-aspherical boundary.
\end{ex}

\begin{ex}
The fundamental group of any countable, locally finite, finite dimensional,
aspherical cell complex occurs as the fundamental group of the
boundary of an aspherical manifold, namely, the boundary
of the regular neighborhood of an codimension $\ge 3$ embedding
of the complex into some manifold (e.g. the Euclidean space). 
If the complex is a manifold, or if it a finite complex, then
the boundary is non-aspherical. (The former follows from
the homotopy sequence of the normal sphere bundles, and the latter
is true because the boundary and the complex
have isomorphic fundamental groups, and if the boundary were
aspherical, it would be a closed manifold
homotopy equivalent to a complex of lower dimension). 
\end{ex}

\begin{ex}
There is many a complete manifold $V$ of $K\le 0$
that deformation retracts onto a compact locally convex
submanifold $S$ of codimension zero (they
are usually called {\it convex-cocompact}\,). The normal
exponential map to $\d S$ identifies $\mathrm{Int}\,S$ with $V$,
and $\d S$ is often not aspherical.
\end{ex}

\begin{ex}
\label{ex: asph codim3 gluing}
Identifying two aspherical $n$-manifolds $N_1$, $N_2$
along diffeomorphic, aspherical, compact, $\pi_1$-injectively
embedded, proper, codimension zero submanifolds $D_i\subset\d N_i$
results in an aspherical manifold $N$, see~\cite{ScoWal}, whose boundary
is a union of $\d N_1-\mathrm{Int}\,D_1$ and $\d N_2-\mathrm{Int}\, D_2$
along $\d D_1\cong \d D_2$. 
Considering the case when $D_i$'s are disks we conclude that
if two manifolds bound aspherical manifolds, 
then so does their connected sum.
The boundary $\d N$ is often non-aspherical, 
e.g. if $N_1$, $N_2$ are compact, $D_i$ is homotopy
equivalent to a complex of dimension $\le n-3$, 
and the inclusions $D_i\hookrightarrow \d N_i$ are not
$\pi_1$-surjective, then
$\pi_1(\d N)$ is a nontrivial amalgamated
product over $\pi_1(D_i)$, 
so $\d N$ is not aspherical by a Mayer-Vietoris argument
in group cohomology. 
\end{ex}

With plenty of examples, we now turn to obstructions. 
In order for $B$ to bound an aspherical manifold, 
a certain covering space of $B$ must bound a 
contractible manifold. In formalizing how this restricts 
the topology of $B$, the following definition is helpful:
given a class of groups $\mathcal Q$, a group 
is {\it anti\,}--$\mathcal Q$ if it admits no nontrivial
homomorphism into a group in $\mathcal Q$. 
Clearly, the class of anti--$\mathcal Q$ groups
is closed under extensions, quotients,
and any group generated by a family of
anti--$\mathcal Q$ subgroups is anti--$\mathcal Q$.

\begin{ex} 
\label{ex: anti-A_n}
Let $\mathcal A_n$ denote the class of fundamental groups of aspherical $n$-manifolds.
Here are some examples of anti--$\mathcal A_n$ groups:
\begin{enumerate}
\item Any group generated by a set of finite order elements.
\item 
Any non-torsion-free simple group. 
\item
Any irreducible lattice in the isometry group of
a symmetric space of rank $\ge 2$ and dimension $>n$~\cite{BesFei02}. 
\item 
Certain finitely presented groups with strong fixed point
properties, see~\cite[Theorem 1.5]{ABJLMS}.
\item
Any non-elementary finitely presented
relatively hyperbolic group
has a finitely presented 
anti--$\mathcal A_n$ quotient,
obtainable by adding finitely many 
relators~\cite[Corrollary 1.6]{AMO}.
\item
Any finitely presented anti--$\mathcal A_n$ group is a
quotient of an anti--$\mathcal A_n$ 
non-elementary hyperbolic group. (The Rips construction 
of~\cite{BO-rips} produces a relatively hyperbolic group
$G$ with $G/K\cong Q$, where $Q$ is the given 
finitely presented anti--$\mathcal A_n$ group,
and $K$ is a quotient of $\mathbb Z_3\ast\mathbb Z_3$.
Here $G$ is anti--$\mathcal A_n$ because so are $Q$ and $K$.) 
\end{enumerate}
\end{ex}

The following summarizes some obstructions that prevent
a manifold from bounding an aspherical one. 

\begin{thm}\label{thm: intro-asp}
If $B$ bounds an aspherical, non-contractible $n$-manifold $N$, 
and $\pi_1(B)$ is anti--$\mathcal A_n$, then $B$ is noncompact,
parallelizable, its $\mathbb Z$-valued intersection
form of vanishes, and  
its $\mathbb Q/\mathbb Z$-valued torsion linking form 
vanishes.  
\end{thm}
\begin{proof}
Since $\pi_1(B)$ is anti--$\mathcal A_n$, the manifold
$B$ also bounds a contractible manifold $W$, whose interior is the
universal cover of $\mathrm{Int}\,N$. Since $W$ is parallelizable,
$B$ is stably parallelizable, and in particular orientable.

Let us show that $B$ is noncompact.
The long exact sequence of the pair gives an isomorphism $H_{n}(W,B)\to H_{n-1}(B)$,
so if $B$ were compact, Poincar\'e duality would imply nontriviality
of $H_{n-1}(B)$, and hence nontriviality of
$H_{n}(W,B)\cong H^0_c(W)$, so there would exist a compactly supported
constant function on the $0$--skeleton of $W$ which is only possible if $W$ were compact, 
but by assumptions $\pi_1(N)$ is nontrivial, and hence infinite,
so $W$ is noncompact.

Since $B$ is non-compact, its stable parallelizability
implies parallelizability.

Recall that the intersection form
$H_k(B)\times H_{n-1-k}(B)\to\mathbb Z$ can be computed
via algebraic intersection numbers.
If $\a$, $\b$ are simplicial cycles in $B$ of complementary dimensions,
then since $B$ bounds a contractible manifold $W$, there are simplicial chains $\hat\a$, $\hat\b$ in $W$ with
$\d \hat\a=\a$ and $\d\hat\b=b$. Choosing $\hat\a$, $\hat\b$ transverse,
and subdividing if necessary,
one observes that the intersection of $\hat\a$, $\hat\b$ is a $1$-chain $\hat c$, and 
the $0$-homology class of $\d\hat c$ equals 
the algebraic intersection number $\iota(\a,\b)$ in $B$.
Now $\iota(\alpha, \beta)$ is a sum of points labelled with $\pm$, and since
it is a boundary, the points come in pairs and for each $+$ there is a $-$, 
so that $\iota(\a,\b)=0$.

If $\tau H_k(B)$ denotes the torsion subgroup of $H_k(B)$,
then the linking form   
$\text{lk}\co \tau H_k(B)\times \tau 
H_{n-2-k}(B)\to \mathbb Q/\mathbb Z$
is defined as follows: given two torsion cycles $\a$, $\b$
the linking coefficient $\text{lk}([\a], [\b])$ equals
$\frac{1}{m}\iota(z,\b)$ modulo $1$, where 
$z$ is any chain in $B$ with $\d z=m\a$.
If $\hat\b$ is a chain in $W$ with $\d \hat\b=\b$,
then $\iota(z,\b)$ equals the intersection number of 
$z$ and $\hat\b$ in $W$.
Since $\a=\d\hat\a$ in $W$, we conclude that $z-m\hat\a$
is a cycle in $W$, which is a boundary, as $W$ is contractible.
The intersection number of any boundary with $\hat\b$ is zero,
hence $\iota(z,\b)$ is divisible by $m$, so
$\text{lk}([\a], [\b])=0$.
\end{proof}

\begin{rmk}
\label{rmk: long exact pair}
If a closed $(n-1)$-manifold $B$ bounds a manifold $W$ with 
$H_{n-1}(W;\mathbb Z_2)=0$, then $W$ is compact and $\d W=B$.
Indeed, $B$ bounds $B\cup\Int(W)$ which is compact by
an argument in the second paragraph of the above proof
giving a compactly supported constant function on
the $0$-skeleton of some triangulation of $B\cup\Int(W)$. 
\end{rmk}

\begin{rmk}
The same proof gives that if $B$ bounds a contractible manifold, then
$B$ is stably parallelizable, and its intersection form
and torsion linking form vanish. By Remark~\ref{rmk: long exact pair} 
if $B$ is compact, then it equals the boundary of a compact contractible manifold, so $B$
is a homology sphere, and conversely, any 
homology sphere bounds a 
topological contractible manifold~\cite{Ker69, Fre-jdg82}.
\end{rmk}

\begin{ex}
The following manifolds do not bound aspherical ones:
\begin{enumerate}
\item The connected sum of lens spaces, because it is a closed
manifold whose fundamental group is anti--$\mathcal A_n$.
\item The product of any manifold with 
$CP^k$ with $k\ge 2$ because it contains a
simply-connected open subset which is not 
parallelizable, namely, the tubular neighborhood of
$CP^k$.
\item The connected sum of any manifold and 
the product of two closed manifolds whose fundamental
groups are anti--$\mathcal A_n$, as the slice inclusions
of the factors in the product have nonzero 
intersection number.
\item
The product of a punctured 
$3$-dimensional lens space and a closed manifold
whose fundamental group is anti--$\mathcal A_n$
(for if $\a$ represents a generator in the first homology
of the lens space, then it links nontrivially with itself,
see e.g.~\cite{PrzYas}, so
its product with the 
closed manifold factor links nontrivially with $\a$).
\item 
Any manifold that contains the manifold in (2), (3), or (4)
as an open subset.
\end{enumerate}
\end{ex}

\section{Reductive groups and regular neighborhoods}
\label{sec: known obstr}

Let $\mathcal{NP}_n$ denote the class
of the fundamental groups of complete
$n$-manifolds of $K\le 0$. 
By the Cartan-Hadamard theorem, 
$\mathcal{NP}_n$ is a subset of $\mathcal{A}_n$. 
The two classes coincide for $n=2$ by the uniformization
and classification of surfaces. On the other hand, for each
$n\ge 4$ there is a closed, locally CAT$(0)$ and hence aspherical, $n$-manifold, 
whose fundamental group is not in $\mathcal{NP}_n$~\cite{DavJan-jdg-hyperb, DJL12}.

As mentioned in the introduction, examples of groups
with strong fixed point properties due to Gromov and Izeki-Nayatani
immediately imply:

\begin{thm}
\label{thm: anti-NP-Gromov}
There is a closed non-aspherical manifold that\vspace{1pt}\newline
\textup{(i)}
bounds a manifold whose interior is covered by a Euclidean space;\vspace{1pt}\newline
\textup{(ii)}
bounds no manifold whose interior 
has a complete metric of $K\le 0$.
\end{thm}
\begin{proof}
Gromov~\cite{Gro-rand03} (cf. \cite[Theorem 1.1]{NaoSil11}) and
Izeki-Nayatani~\cite{IseNay05} showed that that there is many
a finite aspherical simplicial complex $Y$ such that $\pi_1(Y)$ is
anti--$\mathcal{NP}_n$ for all $n$.
(In the work of Izeki-Nayatani $Y$ is the, 
quotient of a Euclidean building by a uniform 
lattice in a $p$-adic symmetric space, while in Gromov's example
$\pi_1(Y)$ is (Gromov) hyperbolic; 
to date there is no explicit example of such a $Y$).
Now embed $Y$ into a Euclidean space as a codimension $\ge 3$
subcomplex and let $N$ be its regular neighborhood. The closed manifold $B:=\d N$
is non-aspherical because the inclusion $B\hookrightarrow N$ is $\pi_1$-injective
and hence the cohomological dimension of $\pi_1(B)$ is bounded above
by the cohomological dimension of $\pi_1(Y)$, which is $\le\dim(Y)$
while $\dim(B)\ge\dim(Y)+2$. 
One can choose the embedding of $Y$ so 
that the universal cover of $\Int(N)$ is simply-connected at infinity 
and hence diffeomorphic to a Euclidean space (in fact, all codimension
$\ge 3$ regular neighborhoods have this property).
Since $\pi_1(B)$ is anti--$\mathcal{NP}_n$,
Theorem~\ref{thm: intro-asp} implies that 
$B$ satisfies (i). 
\end{proof}

Reductive groups (defined in the introduction)
give another source of manifolds satisfying (i)--(ii).

\begin{ex}
\label{ex: elem}
(1) 
Clearly, the property of being reductive is inherited
by quotients, and every group
that is anti--$\mathcal{NP}_n$ for all $n$
is reductive.\newline
(2)
Any finitely generated, virtually nilpotent group
is reductive, see~\cite{Bel-bus} where this is deduced
by combining results of Gromov~\cite{BGS} with the Flat
Torus Theorem. \newline
(3) Any irreducible, uniform lattice in the isometry group
of a symmetric space of $K\le 0$ and real rank $>1$ is reductive,
see~\cite{Bel-bus} where it is deduced from the harmonic map superrigidity.
\end{ex}

\begin{rmk}
\label{rmk: elem}
If $G$ is a group as in Examples~\ref{ex: elem}(2)--(3), then
any torsion-free, quotient of $G$ is the fundamental group
of a closed aspherical $m$-manifold, which is an infranil or
irreducible, locally symmetric of $K\le 0$ of rank $\ge 2$, respectively,
where $m$ is bounded above by the virtual cohomological dimension of $G$. 
If $G$ is virtually nilpotent this follows from~\cite{DekIgo94}, 
and when $G$ is a higher rank lattice one invokes
Margulis's Normal Subgroup Theorem.
\end{rmk}

The following definition helps combine the above examples 
of reductive groups: Given groups $I$, $J$ and
a class of groups $\mathcal Q$ we say that
{\it $I$ reduces to $J$ relative to $\mathcal Q$}
if every homomorphism $I\to Q$ with $Q\in\mathcal Q$   
factors as a composite of an epimorphism $I\to J$ 
and a homomorphism $J\to Q$. 
Clearly if $\mathcal Q=\mathcal{NP}_n$, and $J$ is reductive, then so is $I$, and moreover,
$I$ and $J$ have the same quotients in  $\mathcal{NP}_n$.

\begin{ex}
The class of groups that reduce to $J$ 
relative to $\mathcal Q$ is sizable, e.g.
if $K$ is anti--$\mathcal Q$, then the direct product $K\times J$
 reduces to $J$ relative to $\mathcal Q$;
more generally, the same is true for any amalgamated product  
obtained by identifying a subgroup $A\le J$ with the second factor
of $K\times A$. Furthermore,
a variant of the Rips construction established 
in~\cite{BO-rips} implies that if 
there is an anti--$\mathcal Q$ non-elementary hyperbolic group,
then for any finitely presented group $J$,
there is a non-elementary hyperbolic group $I$
that reduces to $J$ relative to $\mathcal  Q$.
This result applies to $\mathcal Q=\mathcal{NP}_n$ since 
$\mathbb Z_2\ast\mathbb Z_3$ is anti--$\mathcal{NP}_n$
and non-elementary hyperbolic 
(if instead of $\mathbb Z_2\ast\mathbb Z_3$ one uses
the anti--$\mathcal{NP}_n$, torsion-free, hyperbolic groups
of~\cite{Gro-rand03}, then $I$ becomes torsion-free).
\end{ex}

\begin{thm}
\label{thm-intro: codim >2}
Let $B$ be a closed $(n-1)$-manifold such that
$\pi_1(B)$ is reductive and any nontrivial quotient of $\pi_1(B)$ in the class
$\mathcal{NP}_n$ has a finite classifying space of dimension $\le n-3$. 
If $B$ bounds a manifold $N$ such that $\mathrm{Int}\,N$ admits
a complete metric of $K\le 0$, then $N$ is compact, 
and the inclusion $B\to N$ is a $\pi_1$-isomorphism. 
\end{thm}
\begin{proof}
Let $\hat N$ be the covering space of $N$ that corresponds to the image 
of $\pi_1(B)$ under the inclusion $B\hookrightarrow N$. Then $B$ bounds $B\cup\Int(\hat N)$
which by assumption has a classifying space of dimension $\le n-3$, so
$B\cup\Int(\hat N)$ is compact by Remark~\ref{rmk: long exact pair}.
Now covering space considerations imply that $B=\d \hat N$ and $\hat N=N$,
so the inclusion $B\hookrightarrow N$ is $\pi_1$-surjective.

Consider the $\pi_1(N)$-action on the Hadamard manifold that
covers $\mathrm{Int}\,N$. Since $\pi_1(B)$ is reductive, 
the action either stabilizes a horoball or a totally geodesic
submanifold $S$ of dimension $\le n-3$ where the deck-transformation group acts cocompactly.
Hence $\mathrm{Int}\,N$ is diffeomorphic to either the product of $\mathbb R$ and
another manifold, or to the normal bundle of $S/\pi_1(N)$, which is a compact submanifold 
of dimension $\le n-3$. In the latter case $N$ is obtained by attaching an $h$-cobordism
to the tubular neighborhood of  $S/\pi_1(N)$, so the inclusion $B\to N$
is a $\pi_1$-isomorphism. The same is true in the former case by the main result
of~\cite{Bel-bus}. 
\end{proof}

\begin{rmk} The conclusion of Theorem~\ref{thm-intro: codim >2} can be sharpened
as follows. Stallings's embedding up to homotopy type theorem~\cite{Sta-emb-up-to-homot} 
implies that $N$ of Theorem~\ref{thm-intro: codim >2} deformation retracts
to a regular neighborhood of a finite
subcomplex of codimension $q\ge 3$, so an excision argument 
as in~\cite[Theorem 2.1]{Sie-collar} shows that 
$N$ is obtained by attaching an h-cobordism
to the boundary of the regular neighborhood. 
\end{rmk}

If $N$ of Theorem~\ref{thm-intro: codim >2}
is homotopy equivalent to a closed manifold, 
one can say more, which requires the following background.

Let $N$ be a compact $n$-manifold that is homotopy equivalent
to a closed $m$-manifold $M$ such that $q=n-m\ge 3$, and
the inclusion $\d N\to N$ is a $\pi_1$-isomorphism.
The Browder-Casson-Haefliger-Sullivan-Wall embedding 
theorem~\cite[Corollary 11.3.4]{Wal-book}
shows that $N$ is obtained by attaching an h-cobordism
to the boundary of the regular neighborhood
of a PL-embedded copy of $M$.
Abstract regular neighborhoods of $M$ of codimension $q\ge 3$
are $q$-block bundles over $M$, and they are 
classified by homotopy
classes of maps of $M$ into $B\widetilde{PL}_q$,
the classifying space of the semisimplicial group 
$\widetilde{PL}_{q}$~\cite{RouSan-blk-bunI}.
Taking product with the $(q-k)$-cube defines a stabilization map
$B\widetilde{PL}_{k}\to B\widetilde{PL}_{q}$.
Each $q$-block bundle $R$ over $M$ defines a 
spherical fibration over $M$ with structure group 
in $\widetilde{PL}_{q}$ and fiber $S^{q-1}$; up to homotopy the fibration
is the inclusion $\d R\hookrightarrow R$~\cite{RouSan-blk-bunIII}.
Spherical fibrations over $M$ with fiber $S^{q-1}$
are classified by maps $M\to BG_q$, and we say that
{\it the structure group of a fibration reduces to
$\widetilde{PL}_{k\le q}$}  
if its classifying map factors (up to homotopy) through 
$B\widetilde{PL}_{k}\to B\widetilde{PL}_{q}\to BG_q$.
A spherical fibration is {\it linear}\,
if it is isomorphic to a unit sphere bundle 
of a vector bundle. 

\begin{thm} 
\label{thm: spher fibr over M}
Let $M$ be a closed $m$-manifold that is either infranil
or an irreducible, locally symmetric of $K\le 0$ and real rank $\ge 2$. 
Let $B$ be a closed $(n-1)$-manifold such that $q=n-m\ge 3$,
and $\pi_1(B)$ reduces to $\pi_1(M)$ relative to $\mathcal{NP}_n$. 
If $B$ bounds a manifold whose interior admits a complete 
metric of $K\le 0$, then \newline
\textup{(1)} 
$B$ is the total space of a spherical fibration $\b$ over $M$
whose structure group reduces to $\widetilde{PL}_q$. 
\newline
\textup{(2)} 
If $q\ge 4$ and the structure group of $\b$ does not reduce
to $\widetilde{PL}_{q-1}$, or if $q=3$ and $\b$ has no
section, then
$M$ admits a metric of $K\le 0$ and $\b$ is linear. 
\end{thm}
\begin{proof}
Any quotient of $\pi_1(B)$ in $\mathcal{NP}_n$ is also a quotient of 
$\pi_1(M)$ which by Remark~\ref{rmk: elem} 
is the fundamental group of a closed aspherical manifold of 
dimension $\le m$. Thus Theorem~\ref{thm-intro: codim >2}
applies.

If $N$ is as in Theorem~\ref{thm-intro: codim >2},
then the isomorphism $\pi_1(B)\to\pi_1(N)$
factors through a surjection $\pi_1(B)\to\pi_1(M)$,
so $N$ is homotopy equivalent to $M$, and by the preceding
discussion $N$ is obtained by attaching an h-cobordism
to the boundary of the regular neighborhood
of a PL-embedded copy of $M$, which proves (1).
(It is known that $\pi_1(M)$ has trivial Whitehead group,
so the h-cobordism is smoothly trivial if $n\ge 6$
but we do not use this fact here).

Let $H\cong\pi_1(M)$ be the deck-transformation isometry group
of the nonpositively curved metric on the universal cover 
of $\mathrm{Int}\,N$.

If $H$ stabilizes a horoball bounded by a horosphere $S$, 
then $\mathrm{Int}\,N$ is diffeomorphic to $\mathbb R\times S/H$.
In particular, the associated spherical fibration has a section
obtained by sliding along the $\mathbb R$-factor.
If $q\ge 4$, then again the homotopy equivalence
$M\to S/H$ is homotopic to a PL-embedding~\cite[Corollary 11.3.4]{Wal-book},
whose regular neighborhood is a block bundle with structure
group in $\widetilde{PL}_{q-1}$, and hence the same is true
for its product with $\mathbb R$, so that the structure group of
the associated $S^{q-1}$-fibration reduces to
$\widetilde{PL}_{q-1}$.

If $H$ does not stabilize a horoball, then by superrigidity or the Flat
Torus Theorem, $H$ acts cocompactly on a totally geodesic subspace 
whose $H$-quotient is diffeomorphic to $M$ (and in fact homothetic
to $M$ if it has higher rank, or affinely equivalent if $M$ is flat).
The normal exponential map to the subspace is an $H$-equivariant
diffeomorphism, which identifies $\mathrm{Int}\,N$ with the normal bundle
to $M$. Thus the associated spherical fibration is linear.
\end{proof}

\begin{ex}
\label{ex: spherical fibrations}
Let $M$ be as in Theorem~\ref{thm: spher fibr over M}.
In each of the following cases $B$ bounds a manifold covered by 
a Euclidean space, but bounds no 
manifold whose interior admits a complete 
metric of $K\le 0$: 

(1) $B$ is the total space of a linear $S^{q-1}$ linear bundle
over $M$ with no section, where $q\ge 3$ and
$M$ is a non-flat infranilmanifold (e.g. if $M$ is orientable, 
even-dimensional, non-flat infranilmanifold, then the pullback of 
the unit tangent bundle
under a degree one map $M\to S^{m}$ has nonzero
Euler class, and hence no section). 

(2)
\label{thm: cor from bus}
$B=M\times \d C$ where $C$ be a compact contractible manifold 
such that $\pi_1(\d C)$ is a nontrivial group generated by finite order elements.
(In each dimension $\ge 5$ there 
are infinitely many such $C$'s, namely, 
given any finitely presented superperfect group $K$ there is $C$
with $\pi_1(\d C)\cong K$~\cite{Ker69}, 
and the desired infinite family is
obtained by varying $K$ among the free products of 
finitely many superperfect finite groups.)

(3) $B$ is the boundary of a $q$-block bundle over $M$ 
such that $q\ge 4$, the associated spherical fibration
is not linear, and its structure group 
does not reduce to $\widetilde{PL}_{q-1}$. 
(Such block bundles exist below
metastable range. Here is a specific example that works whenever
$q$ is even and $M$ has nonzero Betti numbers in degrees $q$ and $4i$
for some $i>\frac{q}{2}$, as happens for example if $M$ is a torus
of dimension $\ge 4i>2q$.
It is well-known but apparently not recorded
in the literature cf.~\cite{Kle-MO}, that for $q\ge 3$
the classifying space $B\widetilde{PL}_q$
is rationally equivalent to $BO\times BG_q$, which
can be deduced by combining the following results:
$(\widetilde{PL},\widetilde{PL}_{q})\to (G,G_q)$
is a weak homotopy equivalence~\cite[Theorem 1.11]{RouSan-blk-bunIII},  
$G$ is rationally contractible~\cite[Chapter 3A]{MadMil-book}, and 
$BO\to B\widetilde{PL}$ is a rational equivalence, 
see~\cite[Chapter 4C]{MadMil-book} and ~\cite[Corollary 5.5]{RouSan-blk-bunIII}.
Now $BO$ is rationally the 
product of Eilenberg-MacLane spaces corresponding
to Pontryagin classes, and
if $q$ is even, then 
$BG_q$ is rationally $K(\mathbb Q, q)$, detected
by the Euler class. 
By assumption there exist non-torsion cohomology classes $x\in H^q(M)$, 
$y\in H^{4i}(M)$ with $i>\frac{q}{2}$. As in~\cite[Appendix B]{BK-jams}
one can realize nonzero integer multiples of
$x$, $y$ as the Euler class $e$ and the Pontryagin class $p_i$
of a block bundle over $M$.
Its structure
group does not reduce to $\widetilde{PL}_{q-1}$
as $e\neq 0$, so the associated spherical fibration has
no section. The fibration is nonlinear because for linear bundles
$p_i=0$ for $i>\frac{q}{2}$.
Another example can be obtained if $y$ has degree $2q$
and is not proportional to $x^2$ 
by realizing a multiple of $y$ as $p_{q/2}$
an using that $e^2=p_{q/2}$ for linear bundles.) 
\end{ex} 

\section{Ends with injectivity radius going to zero}
\label{sec: new obstr}

Following~\cite{BGS}, we say that 
a subset $S$ of a Riemannian manifold 
{\it has ${Inj}\,{Rad}\to 0$\,} 
if and only if for every $\e>0$
the set of points of $S$ with injectivity radius $\ge \e$ is compact;
otherwise, $S$ {\it has ${Inj}\,{Rad}\not\to 0$}.
For example, by volume comparison
any finite volume complete manifold of $K\le 0$ has
$\mathrm{Inj}\,\mathrm{Rad}\to 0$.
Schroeder proved~\cite[Appendix 2]{BGS} that any complete
manifold of $-1\le K\le 0$ and $\mathrm{Inj}\,\mathrm{Rad}\to 0$
is the interior of a compact manifold with boundary
provided it contains no sequence of totally geodesic, immersed, 
flat $2$-tori whose diameters tend to zero.

Given a compact boundary component $B$ of a manifold $N$,
an {\it end $E$ of $\mathrm{Int}\,N$ that corresponds to $B$} 
is the intersection of $\mathrm{Int}\,N$
with a closed collar neighborhood of $B$; note that 
$E$ is diffeomorphic to $[1,\infty)\times B$.

The {\it cohomological dimension} 
of a group $G$ is denoted $\cd (G)$.

Here is the main result of this paper whose proof is in 
Section~\ref{sec: proofs}.

\begin{thm} 
\label{thm-intro: geom}
Let $N$ be a manifold with compact connected boundary $B$, 
let $E$ be an end of $V:=\mathrm{Int}\,N$ that corresponds to $B$, and let
$H$ be the deck-transformation group of
the universal cover $\widetilde V\to V$ 
corresponding to the image of the 
inclusion induced homomorphism
$\pi_1(E)\to\pi_1(V)$.
If $\widetilde V$ admits
a complete $H$-invariant metric $g$ of $K\le 0$
and $\widetilde V$ contains 
an $H$-invariant horoball, then\newline
\textup{(1)}
$\dim(B)=\mathrm{cd}(H)$ if and only if $B$ is 
incompressible in $N$.\newline
\textup{(2)}
If $\mathrm{Inj}\,\mathrm{Rad}\to 0$ on $E$, then 
$B$ is incompressible in $N$.\newline
\textup{(3)}
If $B$ is incompressible in $N$,
then the universal cover 
of $B$ is homeomorphic \phantom{(3)}
to a Euclidean space.
\end{thm}

\begin{cor} 
\label{cor-intro: tg}
Let $B$ be a closed manifold that
bounds a manifold $N$, and let $E$ be an end of $\mathrm{Int}\,N$
corresponding to $B$. 
If $\pi_1(B)$ is reductive, and 
$\mathrm{Int}\,N$ admits a complete metric of $K\le 0$ and
$\mathrm{Inj}\,\mathrm{Rad}\to 0$ on $E$, then
$B$ is incompressible in $N$.
\end{cor}
\begin{proof}
If $H$ is as in Theorem~\ref{thm-intro: geom}, then since $H$
is reductive, it either stabilizes a horoball or acts cocompactly on a totally geodesic
proper submanifold $S$. In the former case $B$ is incompressible by
Theorem~\ref{thm-intro: geom}. The latter case cannot happen 
because the nearest point projection onto $S$ is $H$-equivariant
and distance nonincreasing, and there is a lower bound for displacement
of elements of $H$ on $S$ while the assumption
$\mathrm{Inj}\,\mathrm{Rad}\to 0$ on $E$ gives
a sequence of points of $V$ whose displacements under some elements 
of $H$ tend to zero.
\end{proof}

\section{Ends with fundamental groups of codimension one}
\label{sec: codim 1}

A small variation on the proof of Theorem~\ref{thm-intro: geom} yields:

\begin{add}
\label{add: codim one}
The part\, \textup{(1)}\ of\, \textup{Theorem~\ref{thm-intro: geom}}\, 
holds if the assumption 
``$\widetilde V$ contains 
an $H$-invariant horoball'' is replaced by ``$\widetilde V$ contains 
an $H$-invariant totally geodesic submanifold of codimension one''.
\end{add}

\begin{cor} 
\label{cor-intro: tors-free quot}
Let $B$ be a closed $(n-1)$-manifold such that
$\pi_1(B)$ is reductive and any nontrivial quotient of $\pi_1(B)$ in the class
$\mathcal{NP}_n$ has cohomological dimension $n-1$. 
If $B$ bounds a manifold $N$ such that $\mathrm{Int}\,N$ admits
a complete metric of $K\le 0$, then $B$ is incompressible in $N$. 
\end{cor}

\begin{proof}
If $H$ is as in Theorem~\ref{thm-intro: geom}, then $H$ cannot
be trivial by Theorem~\ref{thm: intro-asp}, hence
$\mathrm{cd}(H)=\dim(B)$, so $B$ is incompressible in $N$
by Theorem~\ref{thm-intro: geom} and Addendum~\ref{add: codim one}. 
\end{proof}

\begin{ex} 
\label{ex: tf general}  
Corollary~\ref{cor-intro: tors-free quot} applies to the
manifold $M$ as in Theorem~\ref{thm: spher fibr over M} provided
$\pi_1(M)$ has no proper torsion-free quotient, which ensures that
the cohomological dimension of any 
nontrivial torsion-free quotient of $\pi_1(M)$ equals $\dim(M)$.
In fact if $M$ is higher rank,
irreducible, locally symmetric manifold of $K\le 0$,
then  $\pi_1(M)$ has no proper torsion-free quotients
by the Margulis Normal Subgroup Theorem. 
There are also infranilmanifolds whose fundamental group has no proper 
torsion-free quotients. This includes all $3$-dimensional infranilmanifolds
with zero first Betti number, such as the infinite family $2$ 
of~\cite[page 156]{DIKL-3dnil} and the Hantzsche-Wendt flat $3$-manifold,
as well as some higher-dimensional flat manifolds in dimensions 
$p^2-1$~\cite[Theorem 1(9)]{GupSid} and $p^3-p$~\cite[page 29]{Cid} 
for any prime $p$. The fundamental group of the Klein bottle
also has no proper torsion-free quotients.
\end{ex}

\begin{rmk}
If one is not concerned with making sure that $B$
bounds a manifold whose interior is covered by $\mathbb R^n$,
there is a simple recipe for constructing $B$'s
to which Corollaries~\ref{cor-intro: tg} and~\ref{cor-intro: tors-free quot} apply;
for concreteness we focus on the latter.
Start with $\pi_1(M)$ as in Example~\ref{ex: tf general}, 
consider any finitely presented group that reduces to $\pi_1(M)$
relative to $\mathcal{NP}_n$, and realize the
group as the fundamental group of a closed $(n-1)$-manifold $B$,
which is always possible for $n\ge 5$.  Now if $B$ does 
bound a manifold whose interior admits 
a complete metric of $K\le 0$, 
then Corollary~\ref{cor-intro: tors-free quot}
forces $B$ to be incompressible,  and hence
aspherical with $\pi_1(B)\cong\pi_1(M)$.
For instance, $B$ is not incompressible if 
\begin{itemize}
\item
$\pi_1(B)$ has a nontrivial anti-$\mathcal{NP}_n$ subgroup, or 
\item
$B$ is the connected sum $M\# S$, where $S$ is simply-connected and 
not a homotopy sphere (if $M\# S$ were aspherical, its universal
cover would contains a codimension one sphere bounding 
a punctured copy of $S$, so the latter would be contractible).
\end{itemize}
It takes more effort to find $B$ that bounds a
manifold whose universal cover is a Euclidean space but which
cannot bound a manifold whose interior admits a complete metric of $K\le 0$.
Here is an infinite family of such examples.
\end{rmk}

\begin{ex} 
\label{ex: attach circle times homology sphere}
Let $C$ be a compact contractible manifold
as  in Example~\ref{ex: spherical fibrations}(2).
Fix a closed embedded disk $\Delta\subset\d C$.
Take $M$ as in Example~\ref{ex: tf general}, and let $\a$ 
be a homotopically nontrivial
embedded circle in $M$ with a trivial normal bundle; the latter can be 
always arranged by replacing $\a$ with its ``square'' in $\pi_1(M)$.
Attach $M\times [0,1)$ to $S^1\times C$ by identifying the tubular
neighborhood of $\a$ in $M\times\{0\}$
with $S^1\times\Delta$. The result is an aspherical
manifold $N$ with compact boundary $B:=\d N$. 
By Example~\ref{ex: asph codim3 gluing} $B$ is not aspherical.
The arguments of~\cite[Theorem 6.1]{Bel-bus} based on a strong version
of Cantrell-Stallings hyperplane linearization theorem proved in~\cite{CKS}
imply $\mathrm{Int}\,N$ is covered by $\mathbb R^n$. 
Since $\pi_1(B)$ is an amalgamated product of $\pi_1(M)$ and 
$\mathbb Z\times \pi_1(\d C)$ along $\mathbb Z$, where $\pi_1(N)$
is generated by finite order elements,
the only nontrivial
quotient of $\pi_1(B)$ in $\mathcal{NP}_n$ is $\pi_1(M)$, so $B$  
cannot bound a manifold whose interior admits a complete metric of $K\le 0$
by Corollary~\ref{cor-intro: tors-free quot}.
\end{ex}

\section{Proof of Theorem~\ref{thm-intro: geom}}
\label{sec: proofs}

\begin{proof} Figure~\ref{fig} below is meant to accompany the proof of (2),
but it also illustrates the proof of (1) if one ignores the labels on the left.
Let $\widetilde E$ denote a connected $H$-invariant lift of $E$
to $\widetilde V$.

(1)
The ``only if'' direction is trivial: 
the incompressibility and compactness of $B$ implies
$\cd (H)=\dim(B)$, so we focus on the ``if'' direction. 

Let $B_t$ be the fiber of the projection
$E\cong \d E\times [1,\infty)\to [1, \infty)$ over $\{t\}$,
and let $\widetilde B_t\subset\widetilde E$ be the lift of $B_t$.

Compactness of $B_1$ and the fact that $H$ 
preserves a horoball implies that $\widetilde B_1$ is in the $r$-neighborhood 
of a horosphere $H_1$ for some $r$. 
Let $N_{r+1}(H_1)$ be the closed $r+1$-neighborhood 
of $H_1$.

We next show that $\widetilde E$ contains a component 
of $\widetilde V-N_{r+1}(H_1)$. Otherwise, since $\widetilde B_1$ 
separates $\widetilde V$, both components lie in $\widetilde V-\widetilde E$,
so that $\widetilde E\subset N_{r+1}(H_1)$.
Since $\cd(H)=\dim(H_1)$
and $H$ stabilizes $H_1$, the $H$-action on $H_1$
is cocompact, and hence so is the $H$-action on
$N_{r+1}(H_1)$. So $E$ lies is a compact subset of 
$V$, which is a contradiction.

Denote the component of $\widetilde V-N_{r+1}(H_1)$
that lies in $\widetilde E$ by $U_2$; then 
$H_2:=\d U_2$ is a horosphere concentric to $H_1$.

Since $B_t$ separates $E$, the distance from $B_1$ and $B_t$
is an increasing continuous function of $t$, and completeness
of $E$ implies that the function is unbounded. Hence the same is true
for the distance from $\widetilde B_1$ and $\widetilde B_t$,
and therefore there is $s$
such that the distance between $\widetilde B_1$
and $\widetilde B_s$ equals $2r+3$.
Note that $U_2$ contains $\widetilde B_s$.

Since $H$ acts cocompactly on $[1,s]\times \d \widetilde E$,
there is $R$ such that $\widetilde B_s$ lies in the
$R$-neighborhood of $\widetilde B_1$, and hence in the 
$R+r$-neighborhood of $H_1$. Let $H_3$ be a horosphere in 
$U_2$ that bounds the $R+r+1$-neighborhood of $H_1$, so that
$H_3$ is disjoint from $[1,s]\times\d \widetilde E$.

\begin{figure}
\centering
\includegraphics[scale = 0.6]{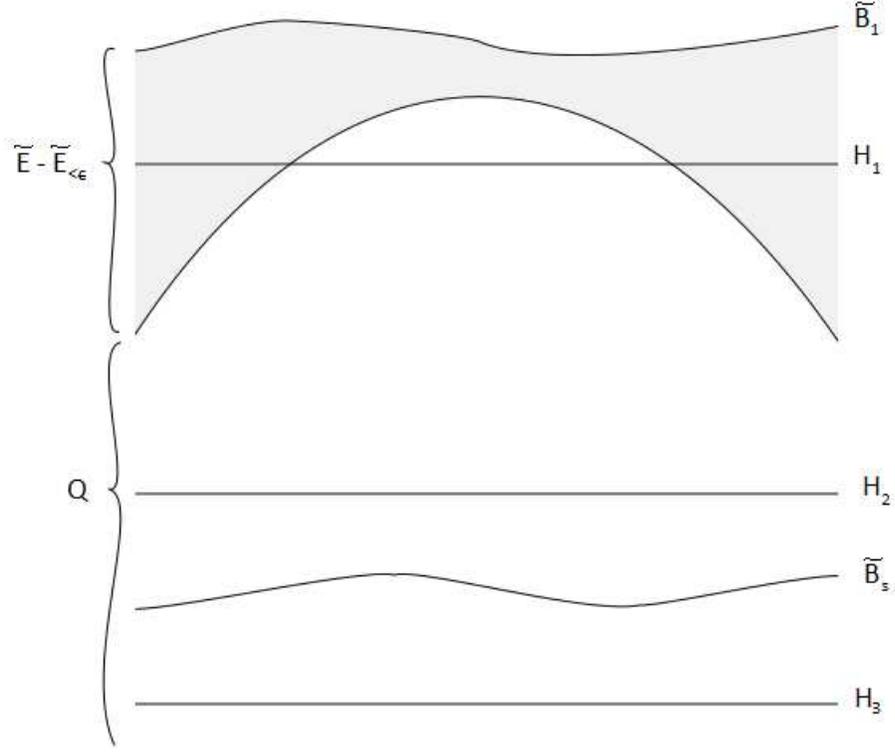}
\caption{Schematic for the proof of Theorem~\ref{thm-intro: geom} (2)}
\label{fig}
\end{figure}

By construction the inclusion
of $\widetilde B_s$ into $\widetilde E$, which is a homotopy equivalence,
factors through the contractible 
region between $H_2$ and $H_3$, so $\widetilde B_s$
is contractible. It follows that $B_s$
is aspherical and $\pi_1$-injectively embedded into $V$, and hence 
$\d N$ is incompressible in $N$.

(2) 
Let $E_{<\e}$ denote the 
set of points of $E$ with injectivity radius $<\e$.
As $\d E$ is compact, $\d E$ and $E_{<\e}$ are disjoint for small enough $\e$,
which we assume henceforth.
Let $\widetilde E_{<\e}$ be the preimage of $E_{<\e}$
under the covering $\widetilde{E}\to E$. 
The set $\widetilde E_{<\e}$ is open and hence locally convex
(for if $d_\g(x)<\e$, then $d_\g(y)<\e$ for all $y$ close 
to $x$, where $d_\g$ denotes the displacement
function for the deck-transformation $\g$).
Since $\mathrm{InjRad}\to 0$ on $E$,
the $H$-action on $\widetilde E-\widetilde E_{<\e}$ is cocompact, so
since $H$ stabilizes a horoball, there is $r$ such that
$\widetilde E-\widetilde E_{<\e}$ lies in an $r$-neighborhood
of a $H$-invariant horosphere, which we denote $H_1$.
Let $H_2$ be the horosphere that
bounds the $r+2$-neighborhood of $H_1$ and
lies in the horoball bounded by $H_1$.
Let $U_2$ be the horoball bounded by $H_2$.
The distance between 
$U_2$ and $\widetilde E-\widetilde E_{<\e}$ is $\ge 2$.

Next we show that $\widetilde E_{<\e}$ is convex and
$U_2\subset \widetilde E_{<\e}$. Let $Q$ be an arbitrary
component of $\widetilde E_{<\e}$. Thus $Q$ is convex
(as a connected, locally convex subset of a Hadamard manifold)
and hence its boundary $\d Q\subset \widetilde E-\widetilde E_{<\e}$
is a topological, properly embedded
submanifold that therefore separates $\widetilde V$.
Consider a ray $\s$ that starts at a point $\s(0)\in Q$
in the $1$-neighborhood of $\d Q$,
and that ends in the center $\s(\infty)$ of $U_2$ at infinity.
As $\s(0)$ lies in the $r+1$-neighborhood of $H_1$, 
the ray $\s$ intersects $H_2$. If $\g\in H$
whose displacement $d_\g$ at $\s(0)$ is $<\e$,
then since $H$ fixes $\s(\infty)$,
nonpositivity of the curvature implies $d_\g(\s(t))<\e$ for all $t$,
so $\s\subset Q$. As $\d Q$ separates $\tilde V$ and
is disjoint from $U_2$, we get $U_2\subset Q$.
Since $Q$ is arbitrary
and $U_2$ is connected, we conclude that 
$Q=\widetilde E_{<\e}$. 

The rest of the proof works verbatim as in (1),
which is a purely topological argument and in particular,
the fact that $U_2$ need not be a horoball in (1), while it is a 
horoball in (2) does not matter.

(3) If $B$ is incompressible, then $B$ is homotopy equivalent
to a horosphere quotient, which is a closed manifold
whose universal cover is diffeomorphic to $\mathbb R^{n-1}$, where $n-1=\dim(B)$.
Thus the universal cover of $B$ is properly homotopy equivalent to $\mathbb R^{n-1}$,
and hence is simply-connected at infinity, and so
homeomorphic to $\mathbb R^{n-1}$
by~\cite{Sta62, Edw-1-conn-at-inf, Wal-1-conn-at-inf, Fre-jdg82}.
\end{proof}

\begin{proof}[Proof of Addendum~\ref{add: codim one}]
The proof of (1) actually works when $H_1$ is any codimension one properly embedded
submanifold whose normal exponential map is a diffeomorphism. 
Then $H_2$, $H_3$ are hypersurfaces equidistant to $H_1$ that bound
$r+1$, $r+R+1$ neighborhoods of $H_1$, respectively.
\end{proof}

\begin{rmk}
In fact, if $B$ is incompressible in $N$, then
$B$ is h-cobordant to a horosphere quotient:
the region between $\widetilde B_1$ and $H_2$ projects
to an h-cobordism embedded in $E$. 
\end{rmk}

\begin{proof}[Proof of Theorem~\ref{thm-intro: no injrad}]
The assertions (1a), (1b), (2) follow from
Corollaries~\ref{cor-intro: tg}, \ref{cor-intro: tors-free quot}, and
Theorem~\ref{thm: spher fibr over M}(1), respectively.
\end{proof}

\small
\bibliographystyle{amsalpha}
\bibliography{enp}

\newcommand{\etalchar}[1]{$^{#1}$}
\def\cprime{$'$}
\providecommand{\bysame}{\leavevmode\hbox to3em{\hrulefill}\thinspace}
\providecommand{\MR}{\relax\ifhmode\unskip\space\fi MR }
\providecommand{\MRhref}[2]{%
  \href{http://www.ams.org/mathscinet-getitem?mr=#1}{#2}
}
\providecommand{\href}[2]{#2}
\begin{thebibliography}{ABJ{\etalchar{+}}09}

\bibitem[ABJ{\etalchar{+}}09]{ABJLMS}
G.~Arzhantseva, M.~R. Bridson, T.~Januszkiewicz, I.~J. Leary, A.~Minasyan, and
  J.~{\'S}wi{\c{a}}tkowski, \emph{Infinite groups with fixed point properties},
  Geom. Topol. \textbf{13} (2009), no.~3, 1229--1263.

\bibitem[AMO07]{AMO}
G.~Arzhantseva, A.~Minasyan, and D.~Osin, \emph{The {SQ}-universality and
  residual properties of relatively hyperbolic groups}, J. Algebra \textbf{315}
  (2007), no.~1, 165--177.

\bibitem[And87]{And-vb}
M.~T. Anderson, \emph{Metrics of negative curvature on vector bundles}, Proc.
  Amer. Math. Soc. \textbf{99} (1987), no.~2, 357--363.

\bibitem[Bela]{Bel-ewp}
I.~Belegradek, \emph{An assortment of negatively curved ends}, to appear in J.
  Topol. Anal., arXiv:1306.1256.

\bibitem[Belb]{Bel-bus}
\bysame, \emph{Obstructions to nonpositive curvature for open manifolds}, to
  appear in Proc. Lond. Math. Soc., arXiv:1208.5220.

\bibitem[BF02]{BesFei02}
M.~Bestvina and M.~Feighn, \emph{Proper actions of lattices on contractible
  manifolds}, Invent. Math. \textbf{150} (2002), no.~2, 237--256.

\bibitem[BGS85]{BGS}
W.~Ballmann, M.~Gromov, and V.~Schroeder, \emph{Manifolds of nonpositive
  curvature}, Progress in Mathematics, vol.~61, Birkh\"auser Boston Inc.,
  Boston, MA, 1985.

\bibitem[BK03]{BK-jams}
I.~Belegradek and V.~Kapovitch, \emph{Obstructions to nonnegative curvature and
  rational homotopy theory}, J. Amer. Math. Soc. \textbf{16} (2003), no.~2,
  259--284.

\bibitem[BK05]{BK-GAFA}
\bysame, \emph{Pinching estimates for negatively curved manifolds with
  nilpotent fundamental groups}, Geom. Funct. Anal. \textbf{15} (2005), no.~5,
  929--938.

\bibitem[BO08]{BO-rips}
I.~Belegradek and D.~Osin, \emph{Rips construction and {K}azhdan property
  ({T})}, Groups Geom. Dyn. \textbf{2} (2008), no.~1, 1--12.

\bibitem[Cid02]{Cid}
C.~F. Cid, \emph{Torsion-free metabelian groups with commutator quotient
  {$C_{p^n}\times C_{p^m}$}}, J. Algebra \textbf{248} (2002), no.~1, 15--36.

\bibitem[CKS12]{CKS}
J.~S. Calcut, H.~C. King, and L.~C. Siebenmann, \emph{Connected sum at infinity
  and {C}antrell-{S}tallings hyperplane unknotting}, Rocky Mountain J. Math.
  \textbf{42} (2012), no.~6, 1803--1862.

\bibitem[DI94]{DekIgo94}
K.~Dekimpe and P.~Igodt, \emph{The structure and topological meaning of
  almost-torsion free groups}, Comm. Algebra \textbf{22} (1994), no.~7,
  2547--2558.

\bibitem[DIKL95]{DIKL-3dnil}
K.~Dekimpe, P.~Igodt, S.~Kim, and K.B. Lee, \emph{Affine structures for closed
  {$3$}-dimensional manifolds with nil-geometry}, Quart. J. Math. Oxford Ser.
  (2) \textbf{46} (1995), no.~182, 141--167.

\bibitem[DJ91]{DavJan-jdg-hyperb}
M.~W. Davis and T.~Januszkiewicz, \emph{Hyperbolization of polyhedra}, J.
  Differential Geom. \textbf{34} (1991), no.~2, 347--388.

\bibitem[DJL12]{DJL12}
M.~Davis, T.~Januszkiewicz, and J.-F. Lafont, \emph{{$4$}-dimensional locally
  {${\rm CAT}(0)$}-manifolds with no {R}iemannian smoothings}, Duke Math. J.
  \textbf{161} (2012), no.~1, 1--28.

\bibitem[Edw63]{Edw-1-conn-at-inf}
C.~H. Edwards, Jr., \emph{Open {$3$}-manifolds which are simply connected at
  infinity}, Proc. Amer. Math. Soc. \textbf{14} (1963), 391--395.

\bibitem[Fre82]{Fre-jdg82}
M.~H. Freedman, \emph{The topology of four-dimensional manifolds}, J.
  Differential Geom. \textbf{17} (1982), no.~3, 357--453.

\bibitem[Gro03]{Gro-rand03}
M.~Gromov, \emph{Random walk in random groups}, Geom. Funct. Anal. \textbf{13}
  (2003), no.~1, 73--146.

\bibitem[GS99]{GupSid}
N.~Gupta and S.~Sidki, \emph{On torsion-free metabelian groups with commutator
  quotients of prime exponent}, Internat. J. Algebra Comput. \textbf{9} (1999),
  no.~5, 493--520.

\bibitem[IN05]{IseNay05}
H.~Izeki and S.~Nayatani, \emph{Combinatorial harmonic maps and discrete-group
  actions on {H}adamard spaces}, Geom. Dedicata \textbf{114} (2005), 147--188.

\bibitem[Ker69]{Ker69}
M.~A. Kervaire, \emph{Smooth homology spheres and their fundamental groups},
  Trans. Amer. Math. Soc. \textbf{144} (1969), 67--72.

\bibitem[Kle]{Kle-MO}
J.~Klein, \emph{Characteristic classes for block bundles}, {M}ath{O}verflow
  (2012), http://mathoverflow.net/a/97484/1573.

\bibitem[MM79]{MadMil-book}
I.~Madsen and R.~J. Milgram, \emph{The classifying spaces for surgery and
  cobordism of manifolds}, Annals of Mathematics Studies, vol.~92, Princeton
  University Press, Princeton, N.J., 1979.

\bibitem[NP]{Pha-sol}
T.~T. Nguyen~Phan, \emph{Nil happens. what about {S}ol?}, arXiv:1207.1734v1.

\bibitem[NS11]{NaoSil11}
A.~Naor and L.~Silberman, \emph{Poincar\'e inequalities, embeddings, and wild
  groups}, Compos. Math. \textbf{147} (2011), no.~5, 1546--1572.

\bibitem[Ont]{Ont-smth-hyp}
P.~Ontaneda, \emph{Pinched smooth hyperbolization}, arXiv:1110.6374v1.

\bibitem[PY03]{PrzYas}
J.~H. Przytycki and A.~Yasukhara, \emph{Symmetry of links and classification of
  lens spaces}, Geom. Dedicata \textbf{98} (2003), 57--61.

\bibitem[RS68a]{RouSan-blk-bunI}
C.~P. Rourke and B.~J. Sanderson, \emph{Block bundles. {I}}, Ann. of Math. (2)
  \textbf{87} (1968), 1--28.

\bibitem[RS68b]{RouSan-blk-bunIII}
\bysame, \emph{Block bundles. {III}. {H}omotopy theory}, Ann. of Math. (2)
  \textbf{87} (1968), 431--483.

\bibitem[Sie69]{Sie-collar}
L.~C. Siebenmann, \emph{On detecting open collars}, Trans. Amer. Math. Soc.
  \textbf{142} (1969), 201--227.

\bibitem[Sta]{Sta-emb-up-to-homot}
J.~R. Stallings, \emph{The embedding of homotopy types into manifolds}, 2000
  version, http://math.berkeley.edu/~stall/embkloz.pdf.

\bibitem[Sta62]{Sta62}
J.~Stallings, \emph{The piecewise-linear structure of {E}uclidean space}, Proc.
  Cambridge Philos. Soc. \textbf{58} (1962), 481--488.

\bibitem[SW79]{ScoWal}
P.~Scott and T.~Wall, \emph{Topological methods in group theory}, Homological
  group theory ({P}roc. {S}ympos., {D}urham, 1977), London Math. Soc. Lecture
  Note Ser., vol.~36, Cambridge Univ. Press, Cambridge, 1979, pp.~137--203.

\bibitem[Wal65]{Wal-1-conn-at-inf}
C.~T.~C. Wall, \emph{Open 3-manifolds which are 1-connected at infinity},
  Quart. J. Math. Oxford Ser. (2) \textbf{16} (1965), 263--268.

\bibitem[Wal99]{Wal-book}
\bysame, \emph{Surgery on compact manifolds}, second ed., Mathematical Surveys
  and Monographs, vol.~69, American Mathematical Society, Providence, RI, 1999,
  Edited and with a foreword by A. A. Ranicki.

\end{thebibliography}

\end{document}